\documentclass[11pt,letterpaper]{amsart}
\usepackage{graphicx}
\usepackage[utf8]{inputenc}
\usepackage{amsmath,amssymb,amsthm}
\newtheorem{theorem}{Theorem}
\newtheorem{lemma}{Lemma}

\newtheorem{corollary}{Corollary}
\newtheorem*{theoremA}{Theorem A}
\newtheorem*{theoremB}{Theorem B}
\theoremstyle{definition}
\newtheorem{definition}{Definition}
\newtheorem{remark}{Remark}
\newtheorem{example}{Example}

\newcommand{\R}{\mathbb{R}}
\newcommand{\C}{{\mathbb{C}}}

\newcommand{\EP}{\widehat{\mathbb{C}}}
\newcommand{\D}{{\mathbb{D}}} 
\newcommand{\BP}{\mathcal{P}} 

\newcommand{\eps}{{\epsilon}}
\newcommand{\absv}[1]{\left\lvert{#1}\right\rvert}
\newcommand{\wto}{\stackrel{\mathrm{w}^*}{\to}}

\newcommand{\thmref}[1]{Theorem~\ref{#1}}
\DeclareMathOperator{\supp}{\mathrm{supp}}

\date{\today}

\begin{document}
\title[Zero Distr.\ of Derivatives of Centering Polynomial Families]%
{Zero Distributions of Derivatives of Polynomial Families
Centering on a Set}
\author{Henriksen, Christian}
\address{%
  Department of Applied Mathematics and Computer Science \\
  Technical University of Denmark}
\email{chrh@dtu.dk} 
\author{Petersen, Carsten Lunde}
\address{%
  Department of Science and Environment \\
  Roskilde University}
\email{lunde@ruc.dk}
\author{Uhre, Eva}
\address{%
  Department of Science and Environment \\
  Roskilde University}
\email{euhre@ruc.dk}
\begin{abstract}
  Suppose $C \subset \C$ is compact.
  Let $q_k$ be a sequence of polynomials of degree $n_k \to \infty$,
  such that the locus of roots of all the polynomials is bounded, and
  the number of roots of $q_k$ in any closed set $L$ not meeting $C$ is uniformly bounded.
  Supposing that $(q_k)_k$ has an asymptotic root distribution $\mu$ 
  we provide conditions on $C$ and $\mu$ assuring the sequence of $m$th derivatives
  $(q_k^{(m)})_k$ also has asymptotic root distribution $\mu$ for any $m\geq 1$. 
  This complements recent results of Totik, \cite{TotikTAMS2019}.
\end{abstract}

\maketitle

\section{Introduction and results}

This paper studies the asymptotic root distribution of 
derivatives of sequences of polynomials. 
In recent years V.~Totik has under quite general conditions (for precise conditions see the discussion trailing Theorem B below) shown the following. 
If the root distributions of a sequence of polynomials converge weak* 
to a compactly supported probability measure $\mu$, 
then also the root distributions of derivatives converge weak* to $\mu$, 
\cite[Theorem 1.6 and Remark 1.12]{TotikTAMS2019}. 
Moreover, he proves an asymptotic Gauss-Lucas theorem:   
If a compact set $C$ contains most of the zeros of a polynomial sequence, 
then any neighborhood of the polynomial convex hull of $C$ contains most of 
the critical points of the polynomials 
\cite[Corollary 1.9]{TotikTAMS2019}.

Recently, Okuyama and Vigny \cite{OkuyamaVigny} have shown 
that preimages of any non-exceptional point (for a definition of exceptional point see below)
under the family of derivatives of iterates of any non-linear polynomial $q$ 
asymptotically equidistributes on the Julia set of $q$. 
Even though the conditions in Totik's theorems are quite general they do not 
apply to the case of families of polynomials obtained by iteration. 
This calls for extending Totik's theorems on the zero distributions of derivatives 
of sequences of polynomials to an even wider class of such sequences. 
This paper does that. 
We build upon and extend Totik's result in \cite{TotikTAMS2019} and obtain 
generalizations, which among other cases apply to sequences of polynomials 
obtained by iteration of a non-linear polynomial.

Moreover, our results have grown out of a quest to relate the properties of extremal families 
of polynomials to their Julia sets and filled-in Julia sets, 
called polynomial convex hull of the Julia sets in potential theory. 
Extremal families of polynomials are central and of great interest  
in many areas of mathematics, to the extent 
where they have become research areas in their own right.  
Examples of extremal families are the sequence of
orthonormal polynomials associated to a compactly supported probability measure on $\C$, 
the sequence of Chebyshev polynomials associated to a compact set in $\C$, 
sequences of Fekete polynomials associated to a compact set, etc. 
A central theme of this ongoing research is to find and prove asymptotic properties 
of the polynomials such as the asymptotic root distribution, 
see e.g.\ \cite{StahlTotik} for a general introduction, and \cite{CHPPOrtho}, \cite{CHPPCheby}, \cite{PU}, \cite{BayraktarEfe}, \cite{CSZ1} for more recent results.

A non-linear polynomial $q(z) \in \C[z]$ defines a dynamical system 
with orbits given by the sequence of iterates 
$q$, $q^2 = q\circ q$, $q^3=q\circ q^2$, $\ldots$.
The sequence of iterates divides $\C$ into the Fatou set $F(q)$, 
where the iterates form a normal family in the sense of Montel
and the compact non-polar complement, the Julia set $J(q)$, 
where the dynamics is chaotic.
The principal questions in holomorphic dynamics are to 
understand the asymptotic behavior of orbits, 
the geometry of $F(q)$ and $J(q)$ and
the geometry of associated parameter spaces. 
See e.g. the books of Milnor \cite{Milnor} or Steinmetz \cite{Steinmetz}
for an introduction to the field. 
We shall also use the exceptional set $E(q)$, 
which is the maximal finite set satisfying $E(q) = q^{-1}(E(q))$. 
Since $J(q)\not=\emptyset$, Montel's theorem implies $E(q)$ is either empty 
or contains at most one point.

For a non-polar polynomially convex compact set $K\subset\C$ 
denote by $\omega_K$ the equilibrium measure on $K$, 
it has support $\supp(\omega_K) =\partial K$. 

Brolin's classical result, \cite{Brolin} relates dynamics to value distribution.
It states that the preimages under $q^k$ of any non-exceptional point 
are asymptotically distributed according to the equilibrium measure on $J(q)$ as $k \to \infty$,
\[ \frac{(q^k)^*\delta_a}{n^k} \wto \omega, \]
Here $\delta_a$ is the Dirac point mass at $a$, and $(q^k)^*\delta_a$
is the pull-back measure, so that $(q^k)^*\delta_a(U)$ is the number
of zeros of $q^k-a$ in $U$ counted with multiplicity.
In particular if $q$ is not of the form $q(z) = \gamma z^d$, 
the zeros of $q^k$ are asymptotically $\omega$-distributed.

The recent results of Okuyama and Vigny (in \cite{Okuyama1} and \cite{OkuyamaVigny})
on value distribution of derivatives of the sequence of iterates,
\[
  \forall\; a \in \C^*\qquad
  \frac{((q^k)^{(m)})^*\delta_a}{n^k-m} \wto \omega, \text{ as $k \to\infty$.
  }
\]
can to a large degree be attributed to the sequence of iterates having a property, 
which we term \emph{centering} and which is related to Totik's notion 
of a set containing most of the zeros.
(see Definition \ref{def:center} and the trailing discussion below).

Here and henceforth $q_k$ is a polynomial with leading term
$\gamma_k z^{n_k}$, and for a sequence of polynomials $(q_k)_k$ 
we suppose that the degrees $n_k$ tend to infinity. 

For $S \subset \C$ we denote by $\BP(S)$ the set of Borel
probability measures supported in $S$.
If $S$ is compact, this space is sequentially compact with respect
to weak-star convergence.

For a non-zero polynomial $q(z) = \gamma \prod_{j=1}^n (z - z_j)$, we
let $\mu_q$ denote the root distribution
\[
  \mu_q = \frac{1}{n} \sum_{j=1}^n \delta_{z_j} =
  \frac{q^*\delta_0}{n} .
\]
In particular, for a set $A \subset \C$,
$n \mu_q(A)$ denotes the number of roots of $q$ in $A$ counted with
multiplicity.

By assumption, the degrees $n_k$ of the sequence of polynomials
$(q_k)_k$ tend to infinity, so for $k$ sufficiently big
$q_k$ is not the zero-polynomial.
Thus we have a sequence $(\mu_{q_k})_k$ of root distributions. 
To alleviate the notation we let $\mu_k = \mu_{q_k}$.
We say the sequence $(q_k)_k$ has asymptotic root distribution $\mu$ 
if and only if $\mu_k \wto \mu$.
We denote by $\mu_k^m$ the sequence of root distributions 
of the sequence of $m$th derivatives $(q_k^{(m)})$.

Let $S = \supp(\mu)$ denote the support of the asymptotic root
distribution.
For $L \in \C$ closed and not meeting $S$, it follows by weak-star
convergence, that the number of zeros of $q_k$ in $L$ counted with
multiplicity is $o(n_k)$.
However, for many families of interest, e.g.~sequences of extremal polynomials and 
iterates of a non-linear polynomial the number of
zeros of $q_k$ in $L$ is uniformly bounded, 
or in other words the number of roots in $L$ is $O(1)$ as $k \to \infty$.
This leads us to the following definition.

\begin{definition}\label{def:center}
  Let $C \subset \C$ be non-empty and compact.
  We say that $(q_k)_k$ \emph{centers} on $C$ if there exist $R > 0$
  and $k_0$ such that
  \begin{enumerate}
    \item All zeros of $q_k$ are contained in the disk $\D(0, R)$
      for all $k \geq k_0$.
    \item If $L \subset \C$ is closed and $L \cap C = \emptyset$,
      then there exists $M = M(L)$ such that the number of zeros of $q_k$ in
      $L$ is bounded by $M$ when $k \geq k_0$.
  \end{enumerate}
\end{definition}

Observe that if $(q_k)_k$ centers on $C$, then $(\mu_k)_k$ 
is precompact, and every limit measure belongs to $\BP(C)$. 

Totik defines a related notion, namely the sequence $(q_k)_k$ 
\emph{has most of its zeros} in $C$ if and only if $\mu_k(C)\to 1$ 
or equivalently $\mu_k(\C\setminus C) = o(1)$. 
This seems like a strong condition, which a priori is hard to check. 
Centering on the other hand may look strong, 
but it holds for extremal sequences such as the sequence 
of orthogonal polynomials for any probability measure with support $C$,  
the sequence of Chebyshev polynomials on $C$ and 
the sequence of iterates of a non-linear polynomial $q$ on 
$C=J(q)$ (see also Examples \ref{iteration_centers_on_J(q)} and 
\ref{extremal_polynomials_centers_on_filled-in_support}). 
Whereas the two conditions, centering on $C$ and having most of its zeros in $C$, 
have the same flavor and certainly have an overlap, 
they are neither equivalent, nor is one stronger than the other. 
Indeed for a quadratic polynomial $q(z) =z^2+c$ with $|c|>2$ 
the sequence of zeros of the iterates $q^k$ centers on $J(q)$, 
but $\mu_k(J(q)) = 0$ for all $k$, see also Example~\ref{iteration_centers_on_J(q)}. 
Thus centering on $C$ does not imply that most of the zeros are in $C$.
Conversely if $C$ is non-empty and compact, $z_0\notin C$ 
and $q_k$ is a polynomial sequence with most of the zeros in $C$ 
and each $q_k$ admits $z_0$ as a zero of order $\lfloor\sqrt{n_k}\rfloor$. 
Then the sequence does not center on $C$ even though most of its zeros are in $C$.

Centering on a polynomially convex compact set 
is a property that is passed on to the sequence of derivatives 
(and by induction the $m$'th derivatives).
We call such properties \emph{hereditary}.

\begin{theoremA}\label{Centering_is_heriditary}
  Suppose $K\subset \C$ is compact and polynomially convex.
  Then centering on $K$ is hereditary.
  That is, if $(q_k)_k$ centers on $K$, then
  the sequence of $m$th derivatives $(q_k^{(m)})_k$ also
  centers on $K$, for $m = 1, 2, \ldots$.
\end{theoremA}

According to the Gauss-Lucas Theorem the root locus of $q_k'$ is 
contained in the convex hull of the root locus of $q_k$, 
as long as $q_k'$ is not the zero polynomial.
Theorem~\ref{Centering_is_heriditary} can be viewed as an asymptotic Gauss-Lucas result 
for derivatives of all orders. 
It should be compared to the asymptotic Gauss-Lucas theorem of Totik, 
\cite[Corollary 1.9]{TotikTAMS2019}, which is non-hereditary. 
Indeed for any quadratic polynomial $q(z) = \lambda z + z^2$ with 
$|\lambda|>4$ all zeros of $q^k$ for any $k\geq 0$ belong to $J(q)$, 
but no critical point belongs to $J(q)$. 

Let $S\subset \C$ be compact and suppose $\mu \in \BP(S)$.
Then the \emph{potential} $p_\mu$ of $\mu$ is defined as
\[
  p_\mu(z) = \int \log \absv{z-w}\,d\mu(w),
\]
following the sign convention of Ransford, see 
\cite[Definition 3.1.1]{Ransford}.
It follows from the definition that $p_\mu$ is subharmonic on $\C$ and harmonic on
$\C \setminus S$.
And additionally,
\[
  p_\mu(z) = \log\absv{z} + O(\absv{z}^{-1}) \text{ as }z\to \infty,
\]
see also \cite[Theorem 3.1.2]{Ransford}.
We say that $p_\mu$ is \emph{nowhere locally constant} in the
open set $U \subset \C$, if the restriction of $p_\mu$ to any
non-empty open subset of $U$ is non-constant.
Additionally, when $p_\mu$ is nowhere locally constant in all of $\C$,
we just say that $p_\mu$ is nowhere locally constant.

Notice that if $p_\mu$ is harmonic in a domain $U$, then it is either
constant on $U$ or nowhere locally constant in $U$.
So on such a domain, nowhere locally constant is the same as
non-constant.

We introduce the following class of measures.
\begin{definition}
  Let $\mu$ be a Borel probability measure with compact support $S \subset \C$ 
  and denote the components of the interior of $S$ by $(W_l)_l$.
  We say that $\mu$ is \emph{noble} if
  we can divide the connected components of $\C \setminus S$
  into two families $(U_i)_i$ and $(V_j)_j$ such that
  \begin{enumerate}
    \item\label{nobel1} the family $(U_i)_i$ contains the unbounded
      component $U_1 = \Omega$, 
    \item\label{nobel2} the potential $p_\mu$ is non-constant on every $U_i$,
    \item\label{nobel3} every $z$ in $\bigcup_j \partial V_j$ is contained in the closure
      of a non-trivial connected component of 
      $\bigcup_i \overline{U}_i \cup \bigcup_l \overline{W}_l$,
    \item\label{nobel4} $\mu(\bigcup_j \overline{V}_j) = 0$ and
    \item\label{nobel5} $m_2(\partial S
      \setminus \left( \bigcup_i \partial U_i
      \cup \bigcup_j \partial V_j
      \cup \bigcup_l \partial W_l \right)) = 0$.
  \end{enumerate}
\end{definition}
Here a non-trivial component is a component containing more
than one point.
Also, $m_2$ denotes the two-dimensional Lebesgue measure, i.e.,
$m_2(A)$ denotes the planar area of $A$.

Care has to be taken with requirement \ref{nobel3}.
If we let $A$ denote the set of points, that are contained in the closure
of a non-trivial connected component of
$\bigcup_i \overline{U}_i \cup \bigcup_l \overline{W}_l$, then
$\bigcup_i \overline{U}_i \cup \bigcup_l \overline{W}_l \subset A
\subset \overline{\bigcup_i U_i \cup \bigcup_l W_l}$, where
both inclusions can be strict.
So requirement \ref{nobel3} is weaker than requiring
every $z \in \bigcup_j \partial V_j$ is contained in the closure
$\overline{\bigcup_i U_i \cup \bigcup_l W_l}$, and we shall show how this can be useful
(in example \ref{ex:circles}).

Since the division of components into $(U_i)_i$ and $(V_j)_j$
is not necessarily unique, we can be more specific and say
$\mu$ is noble with respect to the division $(U_i)_i$ and $(V_j)_j$.

Notice that $\C \setminus S$ has exactly one unbounded component $\Omega$,
and since $p_\mu(z) = \log\absv{z} + o(1)$ as $z\to \infty$,
we have that $p_\mu$ is non-constant on $\Omega$.
However, we do not require $(U_i)_i$ to include all the components
of $\C \setminus S$ where $p_\mu$ is non-constant.

The last piece of notation we need is that of noble polynomial
sequences.

\begin{definition}
  Suppose $\mu \subset \BP(S)$ is noble with respect to the division
  $(U_i)_i$ and $(V_j)_j$, and set $C = \C \setminus \bigcup_i U_i$.
  We say that $(q_k)_k$ is \emph{noble},
  (with respect to $(\mu, (U_i)_i, (V_j)_j)$)
  if $(q_k)_k$ centers on $C$ and $\mu_k \wto \mu$.
\end{definition}

Our next theorem states that nobility is hereditary.
That is, if $(q_k)_k$ is noble, then so is $(q_k')_k$.

\begin{theoremB}
  Let $S \subset \C$ be compact, and suppose $\mu \in \BP(S)$
  is noble with respect to a division $(U_i)_i$, $(V_j)_j$.
  If $(q_k)_k$ is noble with respect to $(\mu, (U_i)_i, (V_j)_j)$,
  then $(q_k')_k$ is also noble with respect to
  $(\mu, (U_i)_i, (V_j)_j)$.
  In particular, for all $m = 0, 1, \ldots$ we have
  \begin{equation}\label{same_limit_measure} 
  \mu_k^m \wto \mu.
  \end{equation}
\end{theoremB}

\begin{remark}
Let us compare Theorem B to Totik's theorem \cite[Theorem 1.6]{TotikTAMS2019}. 
Totik divides the boundary of $S=Supp(\mu)$ into two not necessarily disjoint sets
$$
\partial_{outer}S :=\partial\Omega\qquad\textrm{and}\qquad
\partial_{inner}S :=\partial\left( \bigcup_i U_i\cup \bigcup_j V_j\right).
$$ 
and proves \eqref{same_limit_measure} under the hypothesis that $\mu(\partial_{inner}S) = 0$. 
There are many interesting examples, 
where the hypothesis of \cite[Theorem 1.6]{TotikTAMS2019} is not satisfied, 
but Theorem B applies to yield the desired conclusion. See e.g.~the examples below.
\end{remark}
\begin{remark}
  Whether it is advantageous to include a component of $\C\setminus S$ 
  in the family $(U_i)_i$ or the family $(V_j)_j$
depends on the circumstances, as shown by the following two corollaries.
In the first we take $(U_i)_i$ to include every component,
in the second we take $(U_i)_i = (\Omega)$.
\end{remark}

\begin{corollary}\label{cor:nowhereconst}
  Let $\hat{\Omega} \in \EP$ be a domain containing
  $\infty$, and set $K = \EP \setminus \hat{\Omega}$
  and $S = \partial K = \partial \hat{\Omega}$.
  Suppose $(q_k)_k$ centers on $S$ and $\mu_k \wto \mu$,
  for some $\mu \in \BP(S)$.
  If $p_\mu$ is nowhere locally constant,
  then $(q_k^{(m)})_k$ centers on $S$ and satisfies
  \[
    \mu^m_k \wto \mu,
  \]
  for all $m > 0$.
\end{corollary}

\begin{corollary}\label{cor:nullinnerboundary}
  Let $\hat{\Omega}$, $K$ and $S$ be as in the previous corollary
  and suppose $(q_k)_k$ centers on $K$ and $\mu_k \wto \mu$,
  for some $\mu \in \BP(K)$.
  If $\mu(\partial V) = 0$ for each bounded component $V$ of $\C \setminus S$,
  then $(q_k^{(m)})_k$ centers on $K$ and satisfies
  \[
    \mu^m_k \wto \mu,
  \]
  for all $m > 0$.
\end{corollary}

\begin{example}
  Let $\mu \in \BP(\partial \D)$, and notice that if $\mu$ is not the
  equidistribution on the unit circle $\partial \D$, then $p_\mu$ is
  nowhere locally constant.
  Thus, if $(q_k)_k$ centers on $\partial \D$ and $\mu_k \wto \mu$,
  Corollary \ref{cor:nowhereconst} implies that
  also $(q^{(m)}_k)_k$ centers on $\partial \D$ and
  $\mu^m_k \wto \mu$, for every $m > 0$.
\end{example}

\begin{example}
  Let $K \subset \C$ be polynomially convex, non-polar, and with empty interior.
  It is known (see \cite[Theorem 1.18]{Saff}) that an
  associated sequence of Fekete polynomials $(q_k)_{k > 0}$ has root
  measures $\mu_k$ which satisfy
  \[
    \mu_k \wto \omega,
  \]
  where $\omega$ is the equilibrium measure of $K$.
  It follows from Corollary \ref{cor:nullinnerboundary}
  that the sequence of derivatives
  $(q_k^{(m)})_k$ centers on $K$ and the root distributions
  $\mu^m_k$ satisfy
  \[
    \mu^m_k \wto \omega.
  \]
\end{example}

\begin{example}
  Let $P_c(z) = z^2 + c$.
  The sequence of iterates $P_c, P_c^2, P_c^3, \ldots$
  form a dynamical system for every $c \in \C$.
  The bifurcation locus of this family is the boundary of the celebrated 
  Mandelbrot set $M$; see \cite{DH82} for one of the
  earliest systematic studies.
  Let $\omega$ denote the equilibrium measure of $M$ with support $\partial M$. 
  Since $M$ is polynomially convex and has interior, its interior boundary in the sense of Totik equals its boundary.
  Moreover $\omega(\partial V) = 0$ for any connected component $V$ of the interior of $M$,
  see \cite[proof of Lemma 2]{BuffGauthier} and \cite{Zakeri}. 
  Furthermore in \cite{BuffGauthier} Buff and Gauthier provide examples of sequences 
  $(q_k)_k$ that center on $\partial M$ and satisfies $\mu_k \wto \omega$.
  It follows from Corollary \ref{cor:nullinnerboundary}, 
  that the same holds for the sequence of $m$'th derivatives $(q_k^{(m)})_k$. 
  A result which cannot be derived from \cite[Theorem 1.6]{TotikTAMS2019}.
  
\end{example}

In many families $(q_k)$ studied in the literature, such as the two
previous examples, the root distributions converge weakly to an equilibrium
distribution.
However, given any Borel probability measure $\mu$ on $\C$,
we can construct a sequence
$(q_k)_k$ such that $\mu_k \wto \mu$.
Indeed, it is proven in \cite{Varadarajan} that
if $Z_1, Z_2, \ldots$ are independently sampled according to
$\mu$, then almost surely
$q_k(z) = \prod_{j=1}^k (z - Z_j)$, will have the desired property.
In the following example, we construct sequences of polynomials where
the limit measure is not the equilibrium measure for the support.

\begin{example}\label{ex:cantor}
  Let $S = C \subset [0, \frac{1}{3}] \cup [\frac{2}{3}, 1]$ denote the
  standard middle-third Cantor set.
  Let $q_0$ be any polynomial of degree at least $1$, and define
  recursively
  \[
    q_{k+1} = q_k(3z) q_k(3z-2),~k \geq 0.
  \]
  Then $\mu_k \wto H^s\vert_C$, where $s = \frac{\log 2}{\log 3}$
  and $H^s$ is the $s$-dimensional Hausdorff measure, see the example
  in Section 4.4 of \cite{Hutchinson}.
  The complement of $C$ contains only one component $\Omega$,
  so $H^s\vert_C$ is noble.

  If $d(z, C) < r$ for every root $z$ of $q_k$ and some $r > 0$,
  then $d(z, C) < r/3$ for every root $z$ of $q_{k+1}$.
  It follows that $(q_k)_k$ centers on $C$.

  By Theorem B, the sequence of $m$th-derivatives, $(q_k^{(m)})_k$,
  also centers on $C$ and $\mu_k^m \wto H^s\vert_C$.
\end{example}

In the previous examples, we had $\partial \Omega = S$.
The next example shows, that Theorem B is powerful enough
to handle cases, where this is not the case.
The example also shows that it is relevant to require \ref{nobel3}
in Theorem B, in lieu of the more straightforward but also more restrictive requirement
$\bigcup_j \partial V_j \subset \bigcup_i \partial U_i$.

\begin{example}\label{ex:circles}
  Let $w_j > 0$ be a sequence of weights with
  $\sum_{j=1}^\infty w_j = 1$.
  Let $\nu_r$ denote the equidistribution on $\partial \D(0, r)$,
  and let $\mu = \sum_{j=1}^\infty w_j \nu_{1+1/j}$.
  Then $\mu$ is supported on
  $S = \partial \D(0, 1) \cup \bigcup_{j=1}^\infty \partial \D(0, 1+1/j)$.

  Let $(V_j)_j = (\D(0,1))$, and let $(U_j)_j$ denote the rest of the
  components of $\C \setminus S$.
  Since $\mu(\partial \D(0,1)) = 0$, it is readily verified that
  $\mu$ is noble with respect to this division.

  So if $(q_k)_k$ centers on $C = S \cup \D$ and $\mu_k \wto \mu$
  then the root distributions of the $m$th derivative satisfies
  $\mu^m_k \wto \mu$.
\end{example}

%
\section{Centering}
%

In this section, we study centering in more detail.
We start out with some obvious properties of centering.
Then we will see by example that centering occurs in many 
naturally occurring families of polynomials.
Finally, we will prove a technical theorem, Theorem~\ref{Centering_is_heriditary}, 
which provides conditions under which
the property of centering is hereditary. 
This leads to a proof of Theorem~A.

Clearly, for $(q_k)_k$ to center on $C$, only finitely many of the
polynomials in the sequence can be the zero polynomial.

The following three examples show that many families of polynomials
of great interest center on a compact set $C$.

\begin{example}\label{iteration_centers_on_J(q)}
  Suppose $f$ is a polynomial of degree $d > 1$.
  Then the sequence of iterates $(f^k)_k$, centers on the filled-in
  Julia set $K(f)$.

  This is seen, by first noting that the complement
  $\Omega = \C \setminus K(f)$ is an unbounded domain,
  which is completely invariant by $f$, in the sense that
  $f^{-1}(\Omega) = \Omega$.

  Second, it is known that the Green's function $g_\Omega$ satisfies
  $g_\Omega(f(z)) = d g_\Omega(z)$.
  It implies that if $f$ has no zeros in $\{z: g_\Omega(z) > L\}$,
  then $f^2$ has no zeros in $\{z: g_\Omega(z) > L/d\}$,
  and more generally, $f^k$ has no zeros in $\{z: g_\Omega(z) > L/d^k\}$.

  Actually, with more effort, one can show that the sequence of
  iterates $(f^k)_k$ centers on the Julia set $J = \partial K$,
  unless $f(z) = \gamma z^d$.
\end{example}

\begin{example}\label{extremal_polynomials_centers_on_filled-in_support}
  Widom defines a concept of extremal polynomials associated to a
  compact set $K$, a function $\Phi : \R_+ \cup \{0\} \to \R_+ \cup
    \{0\}$ and measure $\mu$ satisfying certain conditions, see
  \cite{Widom}.
  This is a big class of families of polynomials, that includes the
  orthogonal polynomials associated to $\mu$.
  Widom shows (\cite[Lemma 3]{Widom}) that an extremal
  family centers on $K$.
\end{example}

Having seen examples of centering, we turn our attention to the
consequences of centering.
Recall that $p_\mu$ is harmonic on $G := \C \setminus \supp(\mu)$.
It follows that the function $p_\mu' := 2 \frac{\partial }{\partial
    z} p_\mu$ is holomorphic on $G$.
The following lemma relates
$\frac{1}{n_k}\log\absv{\frac{1}{\gamma_k} q_k(z)}$ to $p_\mu$.

\begin{lemma}\label{lem:weakuniform}
  Suppose all roots of $q_k$ are contained in $\D(0, R)$ and
  $\mu_k \wto \mu$.
  If $U \subset \C$ is open and $q_k(z) \neq 0$ for $z \in U$, for
  $k$ sufficiently large,
  then
  \begin{enumerate}
    \item $\frac{1}{n_k}\log\absv{\frac{1}{\gamma_k} q_k} \to
            p_\mu$ locally uniformly on $U$
    \item $\frac{q_k'}{n_k q_k}$ converges uniformly to the
          holomorphic function
          $p_\mu' = \frac{2 \partial}{\partial z}p_\mu$ on
          compact subsets of $U$.
  \end{enumerate}
\end{lemma}

Note that we can improve the result of the lemma
when $U$ is a deleted neighborhood of infinity.
Since
$p_\mu(z) = \log \absv{z} + o(1)$ and
$p_{\mu_k}(z) = \log \absv{z} + o(1)$
as $z \to \infty$,
$\frac{1}{n_k}\log\absv{\frac{1}{\gamma_k} q_k} \to p_\mu$
and
$\frac{q_k'}{n_k q_k} \to p_\mu'$
uniformly on $U = \{ z \in \C: \absv{z} > R_2\}$
for any $R_2 > R$,
and not just locally uniformly.

\begin{proof}
  Let $A \subset U$ be an arbitrary compact subset.
  Set $B = \overline{\D}(0,R) \setminus U$, then $B$ is also compact.
  Let $S = \supp \mu$ and note that $U\cap S = \emptyset$
  and thus $S \subset B$.

  Define $f : A \times B \to \C$ by $f(z, w) = \log \absv{z - w}$.
  The family $\{f^z : B \to \C\}_{z \in A}$ is uniformly
  equicontinuous and uniformly bounded.

  Write $\frac{1}{\gamma_k} q_k(z) = \prod_{j=1}^{n_k}(z - z_j)$.
  We have
  \[
    \frac{1}{n_k}\log \absv{\frac{1}{\gamma_k} q_k(z)}
    = \frac{1}{n_k}\sum_{j=1}^{n_k}\log\absv{z - z_j}
    = \int f(z, w)\, d\mu_k(w).
  \]
  It follows that if we set $F_k (z) = \frac{1}{n_k}\log
    \absv{\frac{1}{\gamma_k} q_k(z)}$,
  then $(F_k : A \to \C)_k$ is uniformly equicontinuous.

  By weak star convergence of $\mu_k$, $F_k$ converges pointwise to
  $p_\mu$ and by equicontinuity $F_k$ converges to $p_\mu$ uniformly on
  $A$.
  Since $A \subset U$ was arbitrary, we know $F_k$ converges to
  $p_\mu$ locally uniformly on $U$, showing (1).

  Now $\supp \mu \subset S$ does not meet $U$ so $p_\mu$ is harmonic
  on $U$.
  It follows that the holomorphic functions
  $\frac{2\partial}{\partial z} F_k = \frac{q_k'}{n_k q_k}$ converge
  locally uniformly to $\frac{2\partial}{\partial z} p_\mu$ on $U$,
  finishing the proof.
\end{proof}

In order to proceed we need some notation.
For a set $A \subset \C$, we let
$d(z, A) = \inf_{a \in A} \absv{z-a}$
and
$A_{\eps} = \{ z : d(z, A) < \eps \}$.

Suppose $C$ is compact and $\mu \in \BP(C)$,
then $p_\mu$ is harmonic on the complement of $C$.
Thus a critical point of $p_\mu$ is the same as a zero of the holomorphic
function $p_\mu'$.
The identity theorem implies, that if $p_\mu$ is nowhere locally
constant on an open set $U$ not meeting $C$, the number of critical points of $p_\mu$
counted with multiplicity is bounded on every compact subset of
$U$.
The following theorem therefore gives explicit bounds on the number of
zeros of $q_k'$ on compact subsets of $U$. 

\begin{theorem}\label{thmCenter}
  Let $C \subset \C$ be compact and non-empty and let $\mu \in \BP(C)$.
  Suppose $(q_k)_k$ centers on $C$ and $\mu_k \wto \mu$.
  Let $A \subset \C$ be a closed set with $A \cap C = \emptyset$.
  If $p_\mu$ is nowhere locally constant on a neighborhood of $A$,
  then for every $\epsilon > 0$
  \begin{equation}\label{eq:upper}
    \limsup_{k\to\infty}\ (n_k - 1)\mu^1_k(A)
    \leq 
    \limsup_{k\to\infty}\ n_k \mu_k(A_\eps) + M.
  \end{equation}
  and
  \begin{equation}\label{eq:lower}
    \limsup_{k\to\infty}\ (n_k - 1)\mu^1_k(A_\epsilon)
    \geq 
    \limsup_{k\to\infty}\ n_k \mu_k(A) + M,
  \end{equation}
  where $M$ is the number of critical points of $p_\mu$ in $A$.
\end{theorem}

Recall that $(n_k - 1)\mu^1_k(A)$ is the number of zeros
of $q_k'$ in $A$ counted with multiplicity.
The Gauss-Lucas Theorem together with \eqref{eq:upper} shows that under the assumptions
of the theorem, the sequence of derivatives $(q_k')_k$ also 
centers on $C$.

\begin{proof}
  The theorem is a statement about the tail of the sequence
  $(q_k)_k$.
  Since $n_k \to \infty$, we can replace the sequence with 
  $(q_{k_0 + k})_k$ for some $k_0$ big enough so that the degree
  $n_k > 1$, thus ensuring that none of the involved polynomials
  $q_k$ and $q_k'$ is the zero polynomial.

  It then follows from centering, that the locus of the roots of all the
  $q_k$ is contained in some disk $\D(0, R)$.
  Since the roots of $q_k'$ are contained in the convex hull of the
  roots of $q_k$, the roots of $q_k'$ are also contained in $\D(0, R)$.

  Outside $C$, $p_\mu$ is harmonic, and $p_\mu$ has no critical points 
  outside the convex hull of $C$.
  So the set of critical points of $p_\mu$ outside $C$
  is a bounded set.
  Thus, there exists $r > 0$ such that
  the root locus of $(q_k)_k$, the root locus of $(q_k')_k$
  and the critical points of $p_\mu$ outside $C$
  is contained in $\overline{\D}(r)$.
  We can therefore assume $A$ is compact, by replacing it with
  $A \cap \overline{\D}(0, r)$.

  To show \eqref{eq:upper}, we replace $(q_k)_k$ with a subsequence,
  such that the $\limsup$ on the left hand size is realized as a limit.
  If we can show, that for some subsequence $(q_{k_j})$, it holds that
  \[
    (n_{k_j}-1)\mu_{k_j}^1(A) \leq n_{k_j} \mu_{k_j}(A_\epsilon) + M
  \]
  for $k_j$ sufficiently big, then \eqref{eq:upper} follows.

  Write $q_k(z) = \gamma_k \prod_{j=1}^{n_k}(z - z_{k,j})$,
  where we number the roots such that
  $d(z_{k,i}, A) \leq d(z_{k,j}, A)$ when $i < j$.
  For every $j$, there exists a convergent subsequence of  $(z_{k,j})_k$.
  By a standard diagonal argument, we can pass to a subsequence
  such that $(z_{k,j})_k$ is convergent for every $j = 1, 2,\ldots$
  
  Letting $z_j = \lim_{k\to\infty} z_{k,j}$, we have
  \begin{equation}\label{eq:ordering}
    d(z_i, A) \leq d(z_j, A) \text{ when }i < j.
  \end{equation}

  Take $\epsilon > 0$ small enough so that $\overline{A_\epsilon} \cap C = \emptyset$
  and $p_\mu$ is nowhere locally constant in ${A_\epsilon}$.

  By centering on $C$, there is a smallest integer $m$ such that $z_m \notin A$.
  Equation \eqref{eq:ordering} shows that $d(z_j, A) \geq d(z_m, A)$, when
  $j \geq m$, i.e. $z_j \notin A$ when $j \geq m$. 

  Let 
  \[ E = \{z_j : j < m\} \cup \{z \in A: p_\mu'(z) = 0\}. \]

  There is only a finite number of critical points of $p_\mu$ in $A_\epsilon$,
  and the set $E$ is finite.
  Therefore there exists $\delta\in \left(0,\epsilon\right)$ such that
  \begin{enumerate}
    \item $\delta < d(w_1, w_2)$, whenever $w_1 \neq w_2$ and $w_1, w_2 \in E$.
    \item $\delta < d(z_m, A)$.
    \item $\delta < d(z, A)$ for any critical point of
      $p_\mu$ in $A_\epsilon \setminus A$.
  \end{enumerate}

  Choose $k_0$ big enough so that when $k \geq k_0$ and $j \leq m$, we have
  $d(z_{k,j}, z_j) < \delta/4$.
  Then, for every $k\geq k_0$, $d(z_{k,m}, A) > 3\delta/4$ and thus,
  by the numbering of the roots,
  $d(z_{k,j}, A) > 3\delta/4$ for $k \geq k_0$ and $j \geq m$.

  For $k\geq k_0$, there are no zeros of $q_k$ or critical points of
  $p_\mu$ in $\D(z, \delta/2)$, when $z \in A$ and $\D(z, E) > 3\delta/4$.
  Similarly, there are no zeros of $q_k$ or critical points of $p_\mu$ in
  $\D(z, 3\delta/4) \setminus \overline{\D}(z, \delta/4)$, when $z \in E$.

  By compactness, there exists a finite set
  $F \subset A \setminus \bigcup_{z \in E} \D(z, 3\delta/4)$
  such that $\bigcup_{z \in F} \D(z, \delta/2)$ covers
  $A \setminus \bigcup_{z \in E} \D(z, \delta/2)$.

  Lemma \ref{lem:weakuniform} shows that
  $\frac{q_k'}{n_k q_k}$
  converges to $p_\mu'$ on compact subsets of $A_{\delta} \setminus E$.
  Since $p_\mu'$ is not zero on the compact set $\bigcup_{E\cup F} \partial \D(z, \delta/2)$,
  we can assume 
  $\absv{p_\mu' - \frac{q_k'}{n_k q_k}} < \absv{p_\mu'}$ on this set
  for $k\geq k_0$, by increasing $k_0$ if necessary.

  Consider an arbitrary $k \geq k_0$.
  Rouché's theorem then implies, that
  the number of zeros of $q_k'$ minus the number of zeros of $q_k$
  is equal to the number of zeros of $p_\mu'$ in each disk $\D(z, \delta/2)$,
  with $z \in E \cup F$, for $k \geq k_0$.

  Since neither $q_k$ nor $p_\mu'$ has a zero in a disk $\D(z, \delta/2)$,
  when $z \in F$, $q_k'$ has no zeros in such a disk.

  For a disk $\D(z, \delta/2)$ with $z \in E$, the number of zeros of $q_k'$ is equal
  to the number of critical points of $p_\mu$ in the disk, plus
  the number of zeros of $q_k$ in the disk.
  The latter number of zeros equals 
  $\#\{j: z_j \in \D(z, \delta/2)\} = \#\{j: z_j = z\}$.

  We therefore have
  \begin{align*}
    (n_k - 1) \mu_k^1(A)
    & \leq (n_k - 1) \mu_k^1\left(\bigcup_{z \in E \cup F} \D(z, \delta/2)\right) \\
     & = (n_k - 1) \mu_k^1\left(\bigcup_{z \in E} \D(z, \delta/2)\right) \\
     & = m-1 + M \leq n_k \mu(A_\epsilon) + M.
  \end{align*}
  Since this is true for arbitrary $k \geq k_0$ the inequality \eqref{eq:upper}
  follows.

  To see \eqref{eq:lower}, we proceed exactly as before, except we take a
  subsequence realizing the $\limsup$ on the right hand side as a limit.
  We then get
  \begin{align*}
    (n_k - 1) \mu_k^1(A_\epsilon) 
    & \geq (n_k - 1) \mu_k^1\left(\bigcup\limits_{z \in E \cup F} \D(z, \delta/2)\right) \\
    & = (n_k - 1) \mu_k^1(\bigcup_{z \in E} \D(z, \delta/2)) \\
    & = m-1 + M \geq n_k \mu_k(A) + M,
  \end{align*}
  for $k$ sufficiently big, showing the desired inequality.
\end{proof}

If $K$ is polynomially convex, then $p_\mu$ is nowhere
locally constant on the complement of $K$,
for any accumulation point $\mu$ of $\mu_k$.
Thus Theorem A is a corollary of \thmref{thmCenter}.

\begin{corollary}[Theorem A]
  Suppose $K\subset \C$ is compact and polynomially convex.
  Then centering on $K$ is hereditary.
  That is, if $(q_k)_k$ centers on $K$, then
  the sequence of $m$th derivatives $(q_k^{(m)})_k$ also
  centers on $K$, for $m = 1, 2, \ldots$
\end{corollary}
\begin{proof}
  It is enough to prove the statement for the first derivative, since
  the general statement then follows by induction.
  Suppose to the contrary there exists a compact set $A$ which
  doesn't meet $K$, such that
  \[
    \limsup_{k\to\infty} (n_k-1)\mu^1_k(A) = \infty .
  \]
  By passing to a subsequence, we can suppose
  $\mu_k \wto \mu$,
  for some Borel probability measure $\mu$ supported on $K$.
  However, since $p_\mu$ is not locally constant outside $K$,
  the number of critical points $M$ of $p_\mu$ in $A$ is finite,
  so for all $\eps>0$ sufficiently small, we get
  \[
    \infty = \limsup_{k\to\infty} (n_k-1)\mu^1_k(A)
    \leq \limsup_{k\to\infty}\ n_k \mu_k(A_\eps) + M,
  \]
  contradicting that $(q_k)_k$ centers on $K$.
\end{proof}

%
\section{Proof of the Theorem B}
%
In this section we prove Theorem B, and the corollaries
\ref{cor:nowhereconst} and \ref{cor:nullinnerboundary}.

We start by two quite general lemmas.

\begin{lemma}\label{lem:almostIsSure}
  Suppose $u$ is a subharmonic function defined
  on the open set $G \subset \C$.
  Then, for every $z \in G$, it holds that
  \[
    \lim_{\eps \to 0} \frac{1}{\pi \epsilon^2}
    \int_{\D(z, \epsilon)} u\, d m_2
    = u(z).
  \]
  In particular, if $v$ is also subharmonic on $G$
  and satisfies $v \leq u$ $m_2$-almost everywhere in $G$
  then $v \leq u$ in $G$.
\end{lemma}

\begin{proof}
  By the sub-mean property of subharmonic functions,
  there exists $\delta > 0$ such that 
  $u(z) \leq \frac{1}{2\pi}\int_0^{2\pi} u(z+\epsilon e^{i\theta})\, d\theta$,
  for every $0 < \epsilon<\delta$.
  It follows by Fubini's theorem, that when $0 < \epsilon < \delta$,
  we have
  \[
    u(z) \leq \frac{1}{\pi \epsilon^2} \int_{\D(z,\epsilon)} u\, d m_2
          \leq \sup_{0 < \absv{\zeta - z} < \epsilon} u(\zeta).
  \]
  Letting $\epsilon \to 0$, the right hand side will converge to
  $\limsup_{\zeta \to z} u(\zeta)$, a quantity that is bounded above
  by $u(z)$ by upper semicontinuity.
  Therefore we get
  $\lim_{\eps \to 0} \frac{1}{\pi \epsilon^2}
    \int_{\D(z, \epsilon)} u\, d m_2
    = u(z)$ as required.

  Suppose that also $v$ is subharmonic in $G$,
  and $v \leq u$ $m_2$-almost everywhere in $G$.
  We then have, for every $z\in G$, 
  \[ v(z) = 
    \lim_{\eps \to 0} \frac{1}{\pi \epsilon^2}
    \int_{\D(z, \epsilon)} v\, d m_2
    \leq 
    \lim_{\eps \to 0} \frac{1}{\pi \epsilon^2}
    \int_{\D(z, \epsilon)} u\, d m_2
    = u(z).
  \]
\end{proof}

The following lemma is a generalization of Lemma 3 in
\cite{BuffGauthier}.

\begin{lemma}\label{lem:subharmEqualEverywhere} 
  Suppose $u, v$ are subharmonic functions on $\C$,
  such that the generalized Laplacian $\Delta v$ has
  compact support.
  Also suppose $J \subset \C$ is compact.
  Furthermore suppose the connected components of $\C \setminus J$
  are divided into two families $(G_i)_i$ and $(H_j)_j$, such
  that
  \begin{enumerate}
    \item $G_1$ is the unbounded component,
    \item $u = v$ on $\bigcup_i G_i$,
    \item Every $z \in \bigcup_j \partial H_j$ is contained in the closure of a
      non-trivial component of $\bigcup_i \overline{G}_i$,
    \item $\Delta v(\bigcup_j \overline{H}_j) = 0$ and
    \item $m_2(J \setminus \left( \bigcup_i \partial G_i \cup \bigcup_j \partial H_j \right) )
      = 0$.
  \end{enumerate}
  Then $u = v$.
\end{lemma}

Obviously, the conditions in the lemma are closely related to the
requirements for a measure to be noble.

\begin{proof}
  The proof follows the proof of \cite[Lemma 3]{BuffGauthier} very closely.

  If $v \equiv -\infty$, the lemma is trivially true by
  \cite[Theorem 3.5.1]{Ransford}.

  We first show $u \leq v$ on $\C$.
  Since $u = v$ on $G_1$ (by assumption (1) and (2)),
  also $\Delta u$ has compact support.
  Since each $G_i$ is connected, it follows by
  \cite[Theorem 3.8.3]{Ransford}, that for every
  $z \in \overline{G}_i$, we have
  \[
    u(z) = \limsup_{\substack{\zeta \in G_i\\ \zeta \to z}} u(\zeta)
    = \limsup_{\substack{\zeta \in G_i\\ \zeta \to z}} v(\zeta)
    = v(z),
  \]
  and it follows that $u = v$ on each $\overline{G}_i$.
  By assumption (3) and using \cite[Theorem 3.8.3]{Ransford},
  we also have $u=v$ on $\partial H_j$.

  Since $u$ and $v$ agree on the boundary of $H_j$, it follows by
  the Gluing Theorem \cite[Theorem 2.4.5]{Ransford} that
  \[
    w_1 = \begin{cases}
      \max(u,v) & \text{on }H_j, \\
      v & \text{on }\C\setminus H_j.
    \end{cases}
  \]
  is subharmonic.
  Let $\chi : \C \to [0,1]$ be a smooth function with
  compact support, which is equal to $1$ on a neighborhood
  of the support of $\Delta u$ union the support of $\Delta v$
  union $\C\setminus G_1$.
  Then $\Delta \chi$ is supported in $G_1$ and we have
  \[
    \Delta v(\C) = \int_\C \chi\, d \Delta v
    = \int_\C  v\, d \Delta \chi
    = \int_\C  w_1\, d \Delta \chi
    = \int_\C \chi\, d\Delta w_1
    = \Delta w_1(\C).
  \]
  Using that $\Delta v = \Delta w_1$ outside $\overline{H}_j$ and
  assumption (4), we get
  \[
    \Delta w_1(\overline{H}_j)
    = \Delta w_1(\C) - \Delta w_1(\C\setminus \overline{H}_j)
    = \Delta v(\C) - \Delta v(\C\setminus \overline{H}_j)
    = \Delta  v(\overline{H}_j)
    = 0.
  \]
  It follows that $\Delta v = \Delta w_1$ on $\C$.
  By Weyl's Lemma (see \cite[Lemma 3.7.10]{Ransford}),
  $v = w_1 + h$, where $h$ is harmonic.
  Since $v = w_1$ on $G_1$, we get $h \equiv 0$.
  So $v = w_1$ on $\C$, and $u \leq v$ on $H_j$.
  Recalling $u = v$ on $\partial H_j$, we see $u \leq v$
  on $\overline{H_j}$.
  Since this holds for arbitrary $H_j$, we get
  $u \leq v$ on $(\bigcup_i{\overline{G}_i}) \cup (\bigcup_j \overline{H_j})$.
  By assumption (5), this translates to $u \leq v$ $m_2$-almost
  everywhere, and by Lemma \ref{lem:almostIsSure}, we get
  $u \leq v$ on $\C$.

  What remains is to prove $v \leq u$.
  Again, we can deploy the Gluing Theorem and construct a subharmonic
  function $w_2$ by
  \[
    w_2 = \begin{cases}
      \max(u,v) & \text{on } \C \setminus \overline{H}_j \\
      u & \text{on } \overline{H}_j
    \end{cases}
      = \begin{cases}
        v & \text{on }\C \setminus H_j, \\
        u & \text{on } \overline{H}_j.
    \end{cases}
  \]
  For the second inequality, we used that $\max(u,v)=v$.
  Arguing as before, $\Delta w_2(\C) = \Delta v(\C)$.
  We also have $\Delta v(\overline{H}_j) = 0$ and
  $\Delta v = \Delta w_2$ outside $\overline{H}_j$,
  so $\Delta v = \Delta w_2$.
  By Weyl's Lemma $v = w_2 + h$, where $h$ is harmonic,
  and since $v = w_2$ on $G_1$, $v = w_2$.
  So $u = v$ on $H_j$, and since this holds on every $H_j$,
  $u = v$ $m_2$-almost everywhere.
  By Lemma \ref{lem:almostIsSure}, $v \leq u$ on $\C$.
\end{proof}

We are now ready to prove Theorem B.

\begin{proof}[Proof of Theorem B.] 
  We suppose the hypotheses in the theorem are satisfied.
  To alleviate the notation, we set
  $\nu_k = \mu^1_k = \mu_{q_k'}$.

  First note that it is enough to prove the result for the first
  derivative, since the general result then follows by induction.

  Theorem \ref{thmCenter} shows that $(q_k')_k$ centers on $C$,
  so we just need to show that $\nu_k \wto \mu$.
  This is purely a statement on the distribution of zeros, so we can
  suppose that $q_k$ is monic for every $k$ for convenience.

  Since we know that $(q_k')_k$ centers on $C$, the sequence
  $(\nu_k)_k$ is precompact with respect to weak-star
  convergence.
  It is therefore enough to show, that for any convergent
  subsequence, we can find a subsequence that converges to $\mu$.

  Passing to a subsequence such that $\nu_k \wto \nu$,
  we must show $\nu = \mu$.
  Our strategy is to show that $p_\nu = p_\mu$ on every $U_i$
  and every $W_l$
  and then apply
  Lemma \ref{lem:subharmEqualEverywhere}.

  Totik showed, \cite[step IV in the proof of Theorem 1.6]{TotikTAMS2019},
  that $p_\nu = p_\mu$ on each $W_l$.
  
  By centering, $\supp \nu \subset C$.
  We can also suppose that none of the polynomials $q_k'$ is equal to
  the zero polynomial.
  Write $q_k = \prod_{j=1}^{n_k}(z-z_{k,j})$ and
  $q_k' = n_k \prod_{j=1}^{n_k-1} (z - \zeta_{k,j})$.
  We can number the roots so that
  $d(z_{k,1}, C) \geq d(z_{k, 2}, C)
    \geq \cdots \geq d(z_{k, n_k}, C)$
  for all $k$, and similarly for the roots of $q_k'$.
  Arguing as in the proof of Theorem \ref{thmCenter},
  we can pass to a subsequence so that the sequences
  $(z_{k,j})_k$ and $(\zeta_{k,j})_k$ converge for every $j$.
  Set $z_j = \lim_{k\to\infty}z_{k,j}$
  and $\zeta_j = \lim_{k\to\infty} \zeta_{k,j}$.
  By Lemma \ref{lem:weakuniform}, $\frac{1}{n_k} \log \absv{q_k} \to
    p_\mu$ locally uniformly in $\bigcup_i U_i \setminus \{ z_j : j=1,2,
    \ldots\}$.
  Applying the same lemma to $q_k'$, we obtain
  \[
    \lim_{k \to \infty} \frac{1}{n_k} \log \absv{q_k'}
    = \lim_{k\to \infty} \frac{1}{n_k-1} \log \absv{q_k'}
    = \lim_{k \to \infty} \frac{1}{n_k-1} \log
    \absv{\frac{1}{n_k}q_k'}
    = p_\nu
  \]
  locally uniformly on $\bigcup_i U_i \setminus \{ \zeta_j : j=1,2, \ldots\}$.

  Also by Lemma \ref{lem:weakuniform}, for
  $z \in \bigcup_i U_i \setminus \{ z_j : j=1,2, \ldots\}$, we have
  $\frac{1}{n_k}\frac{q_k'(z)}{q_k(z)} = p_\mu'(z) + o(1)$ as $k\to \infty$.
  Hence, if $p_\mu'(z) \neq 0$, we have
  $\frac{1}{n_k}\frac{q_k'(z)}{q_k(z)} = p_\mu'(z)(1+ o(1))$.
  Taking the logarithm and dividing by $n_k$, we get
  $\frac{1}{n_k} \log \absv{q_k'(z)} 
  = \frac{1}{n_k} \log \absv{q_k(z)} + o(1)$.
  In the limit, we get
  $p_{\nu}(z) = p_{\mu}(z)$, for every
  $z \in \bigcup_i U_i \setminus E$,
  where
  \[
    E = \{ \zeta_j : j=1,2, \ldots\} \cup
    \{z_j : j=1,2, \ldots\} \cup
    \{z \in \bigcup_i U_i : p_\mu'(z) = 0\}.
  \]

  Now $U_i \setminus E$ is a connected set, whose closure equals $\overline{U}_i$.
  Hence $U_i \setminus E$ is non-thin 
  at every point of $\overline{U}_i$ by \cite[Theorem 3.8.3]{Ransford}. 
  Thus $p_\mu = p_\nu$ on every set $\overline{U}_i$.
  Similarly, $p_\mu = p_\nu$ on every set $\overline{W}_l$.

  We can therefore
  denote the components of $\bigcup_i U_i \cup \bigcup_l W_l$ by $(G_i)_i$,
  set  $(H_j)_j = (V_j)_j$ and $J = \partial S$
  and apply Lemma \ref{lem:subharmEqualEverywhere} with $v=p_\mu$ og $u=p_\nu$.
  We get that $p_\mu = p_\nu$ on all of $\C$.
  By uniqueness $\mu = \nu$,
  see e.g.~\cite[Theorem 2.10]{Saff}.
\end{proof}

To prove Corollary \ref{cor:nowhereconst}, 
we let $(U_i)_i$ denote the connected components of $\C \setminus S$
(so $(V_j)$ is the empty family).
Assumptions (1) to (6) in Theorem B are readily verified,
so the conclusion of the corollary follows.

To prove Corollary \ref{cor:nullinnerboundary},
let $U_1 = \hat{\Omega} \cap \C$ denote the unbounded component
of $\C\setminus S$, and let $(V_j)_j$ be the rest
of the connected components of $\C\setminus S$.
Since $(q_k)$ centers on $S$, the support of $\mu$
is contained in $S$.
Assumption (1) to (6) in Theorem B now follow easily.

\section{The sequence of antiderivatives}

We have seen, that centering of $(q_k)_k$ on $S$ will,
under certain conditions,
force the derivatives $(q_k')_k$ to also center on $S$.

However, there is no reason to think,
that the centering of $(q_k)_k$ on $S$,
will force a sequence of antiderivatives $(Q_k)_k$ to center on $S$.
If it is the case, the asymptotic root distribution of $(Q_k)_k$
is the same as that of $(q_k)_k$.

\begin{corollary}
  Suppose $S \in \C$ is compact and $m_2(S) = 0$.
  Also suppose $(q_k)_k$ centers on $S$, and $\mu_k \wto \mu$,
  If a sequence of antiderivatives $(Q_k)_k$ centers on $S$,
  then $\mu_{Q_k} \wto \mu$.
\end{corollary}

\begin{proof}
  By Theorem B, any limit point of a subsequence of
  $\mu_{Q_k}$ must be equal to $\mu$.
  Hence the result follows by precompactness of
  $\mu_{Q_k}$.
\end{proof}

Let us now revisit Example \ref{ex:cantor},
and the sequence of polynomials $(q_k)_k$ constructed in this
example.
We claim that no sequence of antiderivatives $(Q_k)_k$ can center on
$K$.

Let $Z_0$ be the zero locus of $q_0$.
Then $Z_{k+1} = h_1(Z_k) \cup h_2(Z_k)$,
when we set $h_1(z) = z/3$ and $h_2(z) = (z+2)/3$.
We can construct an endomorphism $F$ of the compact subsets of
$\C$, by letting $F(C) = h_1(C) \cup h_2(C)$.
This endomorphism is a contraction mapping with respect to the Hausdorff
metric, and $Z_k = F^{k}(Z_0)$ converges geometrically
towards $K$, see the Example 1 in \cite[Section 3.3]{Hutchinson}.

It follows that $\D(0, \frac{1}{12})$,
contains no zeros of $q_k$, for $k$ sufficiently large.
However, by symmetry, $0$ is a critical point of $p_\mu$.
If it were the case, that $(Q_k)_k$ center on $K$,
then $q_k$ should have at least one zero in $\D(0, \frac{1}{12})$
for $k$ sufficiently large, by Theorem \ref{thmCenter}.

\paragraph{Acknowledgments}
The authors would like to thank the
Danish Council for Independent Research |  Natural Sciences for
support via the grant DFF--1026--00267B.

\bibliographystyle{plain}

\bibliography{limitmeasure}

\end{document}